\documentclass[11pt]{article}
\usepackage{amsmath,amsthm,amssymb,amsxtra, graphicx}
\usepackage{amsmath}
\usepackage{graphicx}
\usepackage{amsfonts}
\usepackage{amssymb}
\usepackage{color}
\date{}
\newtheorem{theorem}{Theorem}[section]
\newtheorem{lemma}{Lemma}[section]
\newtheorem{proposition}{Proposition}[section]
\newtheorem{corollary}{Corollary}[section]
\newtheorem{remark}{Remark}[section]

\title{{\large \bf RADEMACHER FUNCTIONS IN WEIGHTED SYMMETRIC SPACES}}
\author{\small{SERGEY ASTASHKIN}}
\begin{document}
\maketitle
\renewcommand{\thefootnote}{\fnsymbol{footnote}}
%

\maketitle

\vspace{5 mm}
\def\bc{\sum_{k=1}^\infty}
\def\symx{\mathrm{Sym}\,(X)}
\def\lx{{\mathcal{M}(X)}}
\def\sublogX{{\mathrm{ln}^{1/2}}}
\def\marf{{M(\varphi)}}
\def\minusloget{\log^{-1/p}(e/t)}
\def\loget{\log^{1/2}(e/t)}
\def\fix{{\varphi_X}}

\begin{abstract}
\noindent The closed span of Rademacher functions is investigated
in the weighted spaces $X(w)$, where $X$ is a symmetric 
space on $[0,1]$ and $w$ is a positive measurable function on $[0,1]$. 
By using the notion and properties of the Rademacher multiplicator space 
of a symmetric space, we give a description of
the weights $w$ for which the Rademacher orthogonal projection is
bounded in $X(w)$. 
\end{abstract}

\maketitle

\vspace{5 mm}

\section{Introduction}

\noindent
We recall that the Rademacher functions on $[0, 1]$ are
defined by $r_k(t)= {\rm{sign}}({\rm{sin}} 2^k\pi t)$ for every
$t\in [0, 1]$ and each $k\in \mathbb{N}$. It is well known that
$\{r_k\}$ is an incomplete orthogonal system of independent random
variables. This system plays a~prominent role in the
modern theory of Banach spaces and operators (see, e.g.,
\cite{DJT}, \cite{JMST}, \cite{MSch} and \cite{Pisier}). 

A~classical result of Rodin and Semenov \cite{RS} states that the sequence
$\{r_k\}$ is equivalent in a~symmetric space $X$ to the unit vector
basis in $\ell_2$, i.e.,
\begin{equation}
\label{Inequ0}
\Big\|\sum_{k=1}^{\infty} a_{k} r_{k}\Big\|_X \asymp
\Big(\sum_{k=1}^{\infty}|a_k|^2\Big)^{1/2}, \quad\, (a_k)\in
\ell_{2},
\end{equation}
if and only if $G\subset X$, where $G$ is the closure of
$L_{\infty}[0,1]$ in the Zygmund space ${\rm Exp}\,L^2[0,1]$. When this
condition is satisfied, the span $[r_k]$ of Rademacher functions
is complemented in $X$ if and only if $X\subset G'$,
where the K\"othe dual space $G'$ to $G$ coincides (with
equivalence of norms) with another well-known Zygmund space $L\log^{1/2}L[0,1].$
This was proved
independently by Rodin and Semenov \cite{RS79} and Lindenstrauss
and Tzafriri \cite[pp.138-138]{LT}. Moreover, 
the condition $G\subset X\subset G'$ (equivalently, complementability of $[r_k]$ in $X$)
is equivalent to the boundedness in $X$ of the orthogonal projection
\begin{equation}
\label{Inequ1}
Pf(t):=\sum_{k=1}^{\infty} c_{k}(f) r_{k}(t),
\end{equation}
where $c_k(f):=\int_0^1 f(u)r_k(u)\,du,$ $k=1,2,\dots$
The main purpose of this paper is to investigate the behaviour of
Rademacher functions and of the respective projection $P$ in the {\it weighted spaces} $X(w)$
consisting of all measurable functions $f$ such that $fw\in X$ with the norm
$\|f\|_{X(w)}:=\|fw\|_X.$ Here, $X$ is a symmetric 
space on $[0,1]$ and $w$ is a positive measurable function on $[0,1]$. 
We make use of the notion of the Rademacher multiplicator space
$\mathcal{M}(X)$ of a symmetric space $X$, which originally arised
from the study of vector measures and scalar functions
integrable with respect to them (see \cite{curbera-1} and \cite{curbera1}). 
For the first time a connection between the space $\mathcal{M}(X)$ and the behavior of Rademacher
functions in the weighted spaces $X(w)$ was observed in \cite{astashkin-curbera-5} when
proving a weighted version of inequality \eqref{Inequ0} (under more restrictive conditions
in the case of $L_p$-spaces it was proved in \cite{veraar}).

To ensure that the operator $P$ is well defined, we have to guarantee that the Rademacher functions 
belong both to $X(w)$ and to its K\"othe dual space $(X(w))'=X'(1/w)$. For this reason, in what follows 
we assume that
\begin{equation}
\label{Inequ2}
L_\infty\subset X(w)\subset L_1.
\end{equation}
This assumption allows us to find necessary and sufficient conditions
on the weight $w$ under which the orthogonal projection $P$ is bounded in 
the weighted space $X(w)$. Moreover, extending above mentioned result of
Rodin and Semenov from \cite{RS} to the {\it weighted} symmetric spaces, we 
show that, in contrast to the symmetric spaces, the embedding $X(w)\supset G$ is 
a stronger condition, in general, than equivalence of the sequence of Rademacher functions in $X(w)$ to the unit vector basis in $\ell_2$.
In the final part of the paper, answering a question from \cite{curbera1},
we present a concrete example of a function $f\in\mathcal{M}(L_1)$, which does not belong 
to the symmetric kernel of the latter space.


\section{Preliminaries}

Let $E$ be a Banach function lattice on $[0,1],$
i.e., if $x$ and $y$ are measurable a.e. finite functions on $[0,1]$ such that $x\in E$ and
$|y|\le |x|$, then $y\in E$ and $\|y\|_E\le \|x\|_E$.
The {\it K\"{o}the dual} of $E$ is the Banach function lattice $E'$ of all
functions $y$ such that $\int_0^1|x(t)y(t)|\,dt<\infty,$ for every
$x\in E,$ with the norm
$$
\|y\|_{E'}:=\sup\Big\{\int_0^1x(t)y(t)\,dt:\,x\in E, \|x\|_E\le 1\Big\}.$$
$E'$ is a subspace of the
topological dual $E^*$. If $E$ is separable we have  $E'=E^*$.
A Banach function lattice $E$ has the {\it Fatou property},
if from $0 \leq x_{n} \nearrow x$ a.e. on $[0, 1]$ 
and $\sup_{n \in {\bf N}} \|x_{n}\|_E < \infty$ it follows that $x \in E$ 
and $\|x_{n}\|_E \nearrow \|x\|_E$.

Suppose a Banach function lattice $E\supset L_\infty.$
By $E_\circ$ we will denote the closure of $L_{\infty}$ in $E.$ Clearly, $E_\circ$
contains the {\it absolutely continuous part} of $E$, that is,
the set of all functions $x\in E$ such that $\lim_{m(A)\to 0}\|x\cdot \chi_A\|_E=0$.
Here and next, $m$ is the Lebesgue measure on $[0,1]$ and 
$\chi_A$ is the characteristic function of a set $A\subset [0,1]$. 

Throughout the paper a {\it symmetric (or rearrangement inveriant) space} $X$ is a
Banach space of classes of measurable functions on [0,1] such that
from the conditions $y^*\le x^*$ and $x\in X$ it follows that $y\in X$ and $\|y\|_X\le\|x\|_X$.
Here, $x^*$ is the decreasing rearrangement of $x$, that is, the
right continuous inverse of its distribution function:
$n_x(\tau)=m\{t\in [0,1]:\,|x(t)|>\tau\}$ Functions $x$ and $y$ are said to be
{\it equimeasurable} if $n_x(\tau)=n_y(\tau)$, for all $\tau>0$. The
{\it K\"{o}the dual} $X'$ is a symmetric space whenever $X$ is symmetric. 
In what follows we assume that 
$X$ is isometric to a subspace of its second K\"{o}the dual $X'':=(X')'$. 
In particular, this holds if $X$ is separable or it has the {\it Fatou property}.
For every symmetric space $X$ the following continuous embeddings hold:
$L_\infty\subset X\subset L_1.$ If $X$ is a symmetric space, $X\ne L_\infty,$ 
then $X_\circ$ is a separable symmetric space.

Important examples of symmetric spaces are Marcinkiewicz, Lorentz and
Orlicz spaces. Let $\varphi\colon[0,1]\to[0,+\infty)$  be a
{\it quasi-concave function}, that is, $\varphi$ increases, $\varphi(t)/t$
decreases and $\varphi(0)=0$. The {\it Marcinkiewicz space} $\marf$  is
the space of all measurable functions  $x$ on [0,1] for which  the
norm
\begin{equation*}
\|x\|_{\marf}=\sup_{0<t\le 1}\, \frac{\varphi(t)}{t}\,\int_0^tx^*(s)
\, ds<\infty.
\end{equation*}
If $\varphi\colon[0,1]\to[0,+\infty)$  is an increasing concave
function, $\varphi(0)=0$, then the {\it Lorentz space} $\Lambda(\varphi)$
consists of all measurable functions $x$ on [0,1] such that
\begin{equation*}
\|x\|_{\Lambda(\varphi)}=\int_0^1x^*(s)\,d\varphi(s) <\,\infty.
\end{equation*}
For arbitrary increasing convex function $\varphi$ we have $\Lambda(\varphi)'=M(\tilde\varphi)$
and $\marf'=\Lambda(\tilde\varphi)$, where $\tilde\varphi(t):=t/\varphi(t)$ 
\cite[Theorems~II.5.2 and II.5.4]{KPS}.
%

Let $M$ be an {\t Orlicz function}, that is, an increasing convex
function on $[0,\infty)$ with $M(0)=0$. The norm of the {\it Orlicz space}
$L_M$ is defined as follows
\begin{equation*}
\|x\|_{L_M}=\inf\bigg\{\lambda
>0:\,\int_0^1M\bigg(\frac{|x(s)|}{\lambda}\bigg)\,ds\,\le\,1\bigg\}.
\end{equation*}
In particular, if $M(u)=u^p,$ $1\le p<\infty,$ we have $L_M=L_p$ isometrically.
Next, by $\|f\|_p$ we denote the norm $\|f\|_{L_p}.$

The {\it fundamental function} of a symmetric space $X$ is the function
$\phi_X(t):=\|\chi_{[0,t]}\|_X$. In particular, we have
$\phi_{\marf}(t)=\phi_{\Lambda(\varphi)}(t)=\varphi(t)$, and
$\phi_{L_M}(t)=1/M^{-1}(1/t)$, respectively.
The Marcinkiewicz $M(\varphi)$ and Lorentz $\Lambda(\varphi)$ spaces
are, respectively, the largest and the smallest symmetric spaces with the
fundamental function $\varphi$, that is, if the fundamental function
of a symmetric space X is equal to $\varphi$, then
$\Lambda(\varphi)\subset X\subset M(\varphi)$.

If $\psi$ is a positive function defined on [0,1], then its lower
and upper dilation indices are
\begin{equation*}
\gamma_\psi := \lim_{t\to 0^+} \frac{\log\big(\sup_{\,0<s\le 1}
\frac{\psi(st)}{\psi(s)}\big)}{\log t}\;\;\mbox{and}\;\;
\delta_\psi := \lim_{t\to +\infty} \frac{\log\big(\sup_{\,0<s\le
1/t} \frac{\psi(st)}{\psi(s)}\big)}{\log t},
\end{equation*}
respectively. Always we have $0\le\gamma_\psi\le\delta_\psi\le 1.$
%
%

In the case when $\delta_\varphi<1$
the norm in the Marcinkiewicz space $\marf$ satisfies the equivalence
\begin{equation*}\label{eq:marz}
\|x\|_{\marf}\asymp\sup_{0<t\le 1}\, \varphi(t)x^*(t)
\end{equation*}
\cite[Theorem II.5.3]{KPS}. 
Here, and throughout
the paper,  $A\asymp B$ means that there exist constants $C>0$ and
$c>0$ such that $c{\cdot}A\le B\le C{\cdot}A$.
%



The Orlicz spaces $L_{N_p}$, $p>0,$ where $N_p$ is an Orlicz function equivalent to the function
$\exp(t^p)-1$, will be of major importance in our study. Usually these are referred as the
Zygmund spaces and denoted by ${\rm Exp}\,L^p.$
The fundamental function of ${\rm Exp}\,L^p$ is equivalent to the function
$\varphi_p(t)=\minusloget$. Since  $N_p(u)$ increases at infinity very rapidly,
${\rm Exp}\,L^p$ coincides with the Marcinkiewicz space $M(\varphi_p)$ \cite{lorentz}. This, together with
the equality $\delta_{\varphi_p}=0<1$, gives
\begin{equation*}
\label{EQ: Norm of x in L_N} \|x\|_{{\rm Exp}\,L^p} \asymp \sup_{0<t\le 1}
x^*(t) \minusloget.
\end{equation*}
In particular, for every $x\in {\rm Exp}\,L^p$ and $0<t\le 1$ we have
\begin{equation}
\label{EQ: Pointwise bound of x in L_N} x^*(t)\le  C\,\|x\|_{{\rm Exp}\,L^p}\
\log^{1/p}(e/t).
\end{equation}
Hence, for a symmetric space $X$, the embedding ${\rm Exp}\,L^p\subset X$ is equivalent to the condition
$\log^{1/p}(e/t)\in X$.

Recall that the Rademacher functions are $r_k(t):=\mathrm{sign}\,\sin(2^k\pi
t)$, $t\in[0,1]$, $k\ge1$. 
The famous Khintchine inequality \cite{Kh} states that, for every $1\le p<\infty$, 
the sequence $\{r_k\}$ is equivalent in 
$L_p$ to the unit vector basis in $\ell_2$. 
As was mentioned in Introduction, Rodin and Semenov extended 
this result to the class of symmetric spaces showing that inequality
\eqref{Inequ0} holds in a symmetric space $X$ if and only if $G\subset X$, 
where $G=({\rm Exp}\,L^2)_\circ.$ \cite{RS}. 
Next, we will repeatedly use
the Khintchine $L_1$-inequality with optimal constants:
\begin{equation}
\label{EQ: L_1-inequaily}
\frac{1}{\sqrt{2}}\|(a_k)\|_{\ell_{2}}\le\Big\|\sum_{k=1}^{\infty} a_{k} r_{k}\Big\|_1 \le
\|(a_k)\|_{\ell_{2}}
\end{equation}
(see \cite{Szarek}), where $\|(a_k)\|_{\ell_{2}}:=(\sum_{k=1}^{\infty}|a_k|^2)^{1/2}$.




The {\it Rademacher multiplicator space} of a symmetric space $X$ is the space $\lx$
of all measurable functions $f\colon[0,1]\to\mathbb R$ such that
$f\cdot \sum_{k=1}^\infty a_kr_k\in X$, for every Rademacher sum $\sum_{k=1}^\infty a_kr_k\in X$. It 
is a Banach function lattice on $[0,1]$ when endowed with the norm
\begin{equation*}
\|f\|_{\lx}=\sup\Bigl\{\Big\|f\cdot\sum_{k=1}^\infty a_kr_k\Big\|_X:\,
 \Big\|\sum_{k=1}^\infty a_kr_k\Big\|_X\le1\Bigr\}.
\end{equation*}
$\lx$ can be viewed as the space of operators given by multiplication by a
measurable function, which are bounded from the subspace $[r_k]$ in X into the whole space $X$.

The Rademacher multiplicator space $\lx$ was firstly considered in
\cite{curbera}, where it was shown that for a broad class of
classical symmetric spaces $X$  the space $\lx$ is not symmetric. This result
was extended in \cite{astashkin-curbera-1} to include all symmetric
spaces such that the lower dilation index $\gamma_{\fix}$ of their
fundamental function $\fix$ is positive. This result
motivated the study of the symmetric kernel $\symx$ of the space
$\lx$. The space $\symx$ consists of all functions $f\in\lx$ such that an arbitrary function $g,$
equimeasurable with $f,$ belongs to $\lx$ as well. The norm in $\symx$ is defined as follows
$$
\|f\|_{\symx}= \sup\|g\|_{\lx},$$ 
where the supremum is taken over all $g$ equimeasurable with $f$. From the definition
it follows that $\symx$ is the largest symmetric space embedded into $\lx$ (see also
\cite[Proposition~2.4]{astashkin-curbera-1}). Moreover, if $X$ is a symmetric space such 
that $X''\supset {\rm Exp}\,L^2$,
then 
$$
\|f\|_{\symx}\asymp \|f^*(t)\loget\|_{X''}$$ 
(see \cite[Proposition~3.1 and Corollary~3.2]{astashkin-curbera-3}). The opposite situation is when the
Rademacher multiplicator space $\lx$ is symmetric. The simplest case of
this situation is when $\lx=L_\infty$. It was shown in
\cite{astashkin-curbera-2}  that $\lx=L_\infty$ if and only if $\loget\not\in X_\circ$. 
Regarding the case when $\lx$ is a symmetric space different from $L_\infty$
see the paper \cite{astashkin-curbera-3}.

We will denote by $\Delta_n^k$ the dyadic intervals of [0,1], that
is, $\Delta_n^k=[(k-1)2^{-n},k2^{-n}]$, where
$n=0,1,\dots$, $k=1,\dots,2^n$; we say that $\Delta_n^k$
has  \textit{rank} $n$.
For any undefined notions we refer the reader to the
monographs \cite{BS}, \cite{KPS}, \cite{LT}.

\vspace{2 mm}

\section{Rademacher sums in weighted spaces}

First, we find necessary and sufficient conditions on the symmetric space $X,$ under which
there is a weight $w$ such that the sequence of Rademacher functions spans $\ell_2$ in $X(w)$.
We prove the following refinement of the nontrivial part
of above mentioned Rodin--Semenov theorem.

\begin{proposition}
\label{prop2}
For every symmetric space $X$ the following conditions are equivalent:

$(i)$ there exists a set $D\subset [0,1]$ of positive measure such that 
\begin{equation}
 \label{ineq1}
\Big\|\bc a_kr_k\cdot \chi_D\Big\|_X\le M\|(a_k)\|_{\ell_2},
\end{equation}
for some $M>0$ and arbitrary $(a_k)\in \ell_2$;

$(ii)$ $X\supset G.$
\end{proposition}
\begin{proof} 

Since implication $(ii)\Rightarrow (i)$ is an immediate consequence of 
the fact that the sequence $\{r_k\}$ spans $\ell_2$ in the space $G$
(see \cite{PZ30} or \cite[Theorem V.8.16]{Zy59}), we need to prove only that $(i)$ implies $(ii).$ 

Assume that \eqref{ineq1} holds. By Lebesgue's density theorem, for sufficiently
large $m\in\mathbb{N}$, we can find a dyadic interval
$\Delta:=\Delta_m^{k_0}=[(k_0-1)2^{-m},k_02^{-m}]$ such that
\begin{equation*}\label{Equation: 21}
2^{-m} = m(\Delta)\ge m(\Delta\cap D)> 2^{-m-1}.
\end{equation*}
Let us consider the set $E=\bigcup_{k=1}^{2^m}E_m^k,$ where $E_m^k$ is
obtained by translating the set $\Delta\cap D$ to the interval
$\Delta_m^{k}$, $k=1,2,\dots,2^m$, (in particular,
$E_m^{k_0}=\Delta\cap D)$. Denote $f_i=r_i\cdot \chi_E$, $i\in
\mathbb{N}$. It follows easily that $|f_i(t)|\le 1$, $t\in [0,1]$,
$\|f_i\|_{2}\ge 1/\sqrt{2}$, and $f_i\to 0$ weakly in $L_2[0,1]$
when $i\to\infty.$ Therefore, by \cite[Theorem~5]{astashkin3}, the
sequence $\{f_i\}_{i=1}^\infty$ contains a subsequence
$\{f_{i_j}\},$ which is equivalent in distribution to the Rademacher
system. The last means that there exists a constant $C>0$ such that
\begin{eqnarray*}
C^{-1}m\biggl\{t\in[0,1]:\biggl|\sum_{j=1}^{l} a_j
r_{j}(t)\biggr|>Cz\biggr\} &\le&
m\biggl\{t\in[0,1]:\biggl|\sum_{j=1}^{l} a_j
f_{i_j}(t)\biggr|>z\biggr\}\\
&\le& C m\biggl\{t\in[0,1]:\biggl|\sum_{j=1}^{l} a_j
r_j(t)\biggr|>C^{-1}z\biggr\}
\end{eqnarray*}
for all $l\in{\mathbb N},$ $a_j\in{\mathbb R}$, and
$z>0.$ Hence, by the definition of $r_j$ and $f_j,$ for every
$n\in{\mathbb N}$ we have
\begin{eqnarray*}
&& C^{-1}m\biggl\{t\in[0,1]:\biggl|\sum_{j=m+1}^{m+n}
r_{j}(t)\chi_{[0,2^{-m}]}(t)\biggr|>Cz\biggr\} \\
&&\qquad\qquad\qquad\qquad\le
m\biggl\{t\in[0,1]:\biggl|\sum_{j=m+1}^{m+n} f_{i_j}(t) \chi_\Delta(t)\biggr|>z\biggr\}\\
&&\qquad\qquad\qquad\qquad\le C
m\biggl\{t\in[0,1]:\biggl|\sum_{j=m+1}^{m+n}
r_j(t)\chi_{[0,2^{-m}]}(t)\biggr|>C^{-1}z\biggr\},
\end{eqnarray*}
whence
\begin{equation}
\label{Equation: special}
\Big\|\sum_{j=m+1}^{m+n} r_{i_j}\chi_{\Delta\cap D}\Big\|_X\ge
\alpha\Big\|\sum_{j=m+1}^{m+n} r_j\chi_{[0,2^{-m}]}\Big\|_X,
\end{equation}
where $\alpha>0$ depends only on the constant $C$ and on the space
$X.$ 

Now, assume that $(ii)$ fails, i.e., $X\not\supset G.$ Then,
by \cite[inequality (2) in the proof of Theorem 1]{astashkin-curbera-2},
there exists a constant $\beta>0$, depending only on
$X$, such that for every $m\ge0$ there exists $n_0\ge1$ such that,
if $n\ge n_0$ and $\Delta$ is an arbitrary dyadic interval of rank
$m$, we have
\begin{equation*}
\Big\|\chi_\Delta\sum_{i=m+1}^{m+n}r_i\Big\|_X\ge \beta\Big\|\sum_{i=1}^{n}r_i\Big\|_X.
\end{equation*}
From this inequality with $\Delta=[0,2^{-m}]$
and inequality \eqref{Equation: special} it follows that, for $n$ large enough,

$$
\Big\|\sum_{j=m+1}^{m+n} r_{i_j}\chi_{D}\Big\|_X\ge
\Big\|\sum_{j=m+1}^{m+n} r_{i_j}\chi_{\Delta\cap D}\Big\|_X\ge
\alpha \beta\Big\|\sum_{j=1}^{n} r_j\Big\|_X.
$$
Combining the latter inequality together with \eqref{ineq1} we deduce  
$$
\frac{1}{\sqrt{n}}\Big\|\sum_{j=1}^{n} r_j\Big\|_X\le \frac{M}{\alpha \beta}
$$
for all $n\in\mathbb{N}$ large enough. At the same time, as it follows from
the proof of Rodin--Semenov theorem \cite{RS}, the last condition is equivalent to the embedding $X\supset G.$ 
This contradiction concludes the proof.
\end{proof}

\begin{corollary}
\label{prop2}
Suppose $X$ is a symmetric space. Then, $X\supset G$ if and only if there exists a weight $w$ such that 
the sequence $\{r_k\}$ spans $\ell_2$ in $X(w)$.
\end{corollary}
\begin{proof} 
If $\{r_k\}$ spans $\ell_2$ in $X(w)$ for some weight $w$,  we have
$$
\Big\|\bc a_kr_k\cdot w\Big\|_X\le C\|(a_k)\|_{\ell_2}.$$
Since $w(t)>0$ a.e. on $[0,1]$, there is a set $D\subset [0,1]$ of positive measure
such that inequality  \eqref{ineq1} holds for some $M>0$ and arbitrary $(a_k)\in \ell_2$.
Applying Proposition \ref{prop2}, we obtain that $X\supset G.$ The converse is obvious, and so
the proof is completed.
\end{proof}

Corollary \ref{prop2} shows the necessity of the condition  $X\supset G$ in the following
main result of this part of the paper.

\begin{theorem}
\label{Th1}
Let $X$ be a symmetric space such that $X\supset G$ and let a positive measurable function $w$ on $[0,1]$ 
satisfy condition \eqref{Inequ2}. Then we have

\begin{itemize}
\item[{\rm(i)}] The sequence $\{r_k\}$ spans $\ell_2$ in $X(w)$ if and only if $w\in \lx,$
where $\lx$ is the  Rademacher multiplicator space of $X;$

\item[{\rm(ii)}] $X(w)\supset G$  if and only if $w\in \symx,$
where $\symx$ is the symmetric kernel of $\lx.$
\end{itemize}
\end{theorem}

The part {\rm(i)} of this theorem was actually obtained in \cite[p.~240]{astashkin-curbera-5}.
However, for the reader's convenience we provide here its proof. But we begin with
the following technical result, which will be needed us to prove the part {\rm(ii)}.

\begin{lemma}
 \label{aux1}
Let $Y$ be a symmetric space and $w$ be a positive measurable function on $[0,1].$ Suppose
the weighted function lattice $Y(w^*)$ contains an unbounded decreasing positive function $a$ on $(0,1]$.
Then $(Y(w))_\circ=Y_\circ(w).$
\end{lemma}
\begin{proof} 
 Since $(wa)^*(t)\le w^*(t/2)a(t/2),$ $0<t\le 1,$ \cite[\S\,II.2]{KPS} and, by assumption,
$w^*a\in Y$, we have $wa\in Y.$ Equivalently, $a\in Y(w).$ 

Let $y\in (Y(w))_\circ.$ By definition, there is a sequence $\{y_k\}\subset L_\infty$
such that 
\begin{equation}
\label{EQ: Lemma1}
\lim_{k\to\infty}\|y_kw-yw\|_Y=0.
\end{equation}
Show that $y_kw\in Y_\circ$ for every $k\in\mathbb{N}.$

Since $a$ decreases, for arbitrary $A\subset [0,1]$ and every (fixed) $k\in\mathbb{N}$ we have
$$
\|y_kw\chi_A\|_Y\le \|y_k\|_\infty\|w^*\chi_{(0,m(A)]}\|_Y\le 
\frac{\|y_k\|_\infty}{a(m(A))}\|w^*a\|_Y.$$
Noting that the right hand side of this inequality tends to $0$ as $m(A)\to\infty,$
we get
$$
\lim_{m(A)\to 0}\|y_kw\chi_A\|_Y=0,$$
whence $y_kw\in Y_\circ$, $k\in\mathbb{N}.$ Combining this with \eqref{EQ: Lemma1},
we infer that $yw\in Y_\circ$ or, equivalently, $y\in Y_\circ(w).$

To prove the opposite embedding, assume that $y\in Y_\circ(w).$ Then 
\begin{equation}
\label{EQ: Lemma2}
\lim_{k\to\infty}\|y_k-yw\|_Y=0
\end{equation}
for some sequence  $\{y_k\}\subset L_\infty$.
From hypothesis of lemma it follows that $Y\ne L_\infty.$ Therefore, for 
arbitrary $A\subset [0,1]$ and each $k\in\mathbb{N}$
$$
\|y_k/w\cdot \chi_A\|_{Y(w)}=\|y_k\chi_A\|_{Y}\to 0\;\;\mbox{as}\;\;m(A)\to 0.$$
Hence, $y_k/w\in (Y(w))_\circ,$ $k\in\mathbb{N}$. Since
$\|y_k/w-y\|_{Y(w)}=\|y_k-yw\|_Y,$
from \eqref{EQ: Lemma2} it follows that $y\in (Y(w))_\circ.$
\end{proof}

\begin{proof}[Proof of Theorem \ref{Th1}]
{\rm(i)} Since $X\supset G$, equivalence \eqref{Inequ0} holds. At first, assume that 
$w\in \lx.$ Then, by definition of the norm in $\lx,$ we have
\begin{equation}
\label{eq200}
\|w\|_\lx\asymp \sup\Big\{\Big\|w\cdot \bc a_kr_k\Big\|_X:\,\|(a_k)\|_{\ell_2}\le 1\Big\}.
\end{equation}
Therefore,
$$
\Big\|\bc a_kr_k\Big\|_{X(w)}= \Big\|w\cdot \bc a_kr_k\Big\|_X\le \|w\|_\lx\|(a_k)\|_{\ell_2}
$$
for every $(a_k)\in\ell_2.$ On the other hand, from embeddings  \eqref{Inequ2}
and inequality \eqref{EQ: L_1-inequaily} it follows that
$$
\Big\|\bc a_kr_k\Big\|_{X(w)}\ge c\Big\|\bc a_kr_k\Big\|_{1}\ge \frac{c}{\sqrt{2}}\|(a_k)\|_{\ell_2}.
$$
As a result we deduce that  $\{r_k\}$ spans $\ell_2$ in $X(w)$.

Conversely, if 
$$
\Big\|\bc a_kr_k\Big\|_{X(w)}\asymp \|(a_k)\|_{\ell_2},$$
from \eqref{eq200} we obtain that $\|w\|_\lx<\infty,$ i.e., $w\in\lx.$

{\rm(ii)} Assume that $w\in\symx.$ Then, taking into account the properties of the symmetric kernel
$\symx$ (see Preliminaries or \cite[Corollary~3.2]{astashkin-curbera-3}) we have 
$w^*(t)\log^{1/2}(e/t)\in X''.$ 
Let us prove that
\begin{equation}
\label{eq201}
{\rm Exp}\,L_2\subset X''(w).
\end{equation}
Given $x\in {\rm Exp}\,L_2$, by  \cite[Theorem 2.7.5]{BS}, there exists a~measure-preserving 
transformation $\sigma$ of $(0, 1]$ such that $|x(t)|= x^{*}(\sigma(t)).$ 
Applying inequality \eqref{EQ: Pointwise bound of x in L_N} and a well-known property of
the rearrangement of a measurable function (see e.g. \cite[\S\,II.2]{KPS}), we have
$$
(wx)^*(t)=\left(wx^{*}(\sigma)\right)^*(t)\le C\left(w\log^{1/2}(e/\sigma(\cdot))\right)^*(t)
\le Cw^*(t/2)\log^{1/2}(2e/t),\;\;0<t\le 1.$$
Therefore, $wx\in X''$ or, equivalently, $x\in X''(w),$ and \eqref{eq201} is proved.
Hence, $G=({\rm Exp}\,L_2)_\circ\subset (X''(w))_\circ.$ Since $\log^{1/2}(e/t)\in X''(w^*)$,
we can apply Lemma \ref{aux1}, and so, by \cite[Lemma~3.3]{As09},
$$
G\subset (X'')_\circ(w)=X_\circ(w)\subset X(w).$$

Now, let $X(w)\supset G.$ We show that $X(w^*)\supset G.$
In fact, let $\tau$ be a~measure-preserving transformation of $(0, 1]$
such that $w(t)= w^{*}(\tau(t))$ \cite[Theorem 2.7.5]{BS}.
Suppose $x\in G.$ Since $x(\tau)$ and $x$ are equimesurable functions, 
we have $x(\tau)\in G$ and $\|x(\tau)\|_G=\|x\|_G.$ Therefore,
$$
\|x(\tau)w^{*}(\tau)\|_X=\|x(\tau)w\|_X\le C\|x\|_G.$$
Then, $\|x(\tau)w^{*}(\tau)\|_X=\|xw^{*}\|_X$, because 
$X$ is a symmetric space, and from the
preceding inequality we infer that $\|xw^{*}\|_X\le C\|x\|_G$.
Thus, $x\in X(w^*),$ and the embedding $X(w^*)\supset G$ is proved.
Passing to the second  K\"othe dual spaces, we obtain:
 $X''(w^*)\supset G''={\rm Exp}\,L^2.$ Hence,  $\log^{1/2}(e/t)\in X''(w^*)$
or, equivalently, $w\in\symx$ (as above, see Preliminaries or
\cite[Corollary~3.2]{astashkin-curbera-3}), and the proof is complete.
\end{proof}

\vspace{2 mm}

%
By Rodin-Semenov theorem \cite{RS}, the sequence
$\{r_k\}$ is equivalent in a~symmetric space $X$ to the unit vector
basis in $\ell_2$ if and only if $X\supset G.$ In contrast to that
from Theorem \ref{Th1} we immediately deduce the following result. 

\begin{corollary}
 Suppose $X$ is a symmetric space such that $\symx\ne \lx.$
Then, for every $w\in \lx\setminus \symx$ the Rademacher functions
span $\ell_2$ in $X(w)$ but $X(w)\not\supset G.$
\end{corollary}

By \cite[Theorem~2.1]{astashkin-curbera-1}, $\symx\ne \lx$ ( and therefore there is 
$w\in \lx\setminus \symx$) whenever the lower dilation index of the fundamental function
$\phi_X$ is positive. In particular, it is fulfilled for $L_p$-spaces, $1\le p<\infty.$
The condition $\gamma_{\phi_X}>0$ means that the space $X$ is situated ``far'' from the minimal
symmetric space $L_\infty.$ Now, consider the opposite case when a symmetric space is  
``close'' to $L_\infty$. Then the Rademacher
multiplicator space $\lx$ may be symmetric (equivalently, it coincides with its symmetric kernel). 
Since the space $\symx$ has an explicit
description (see Preliminaries), in this case we are able to state a sharper result.
For simplicity, let us consider only Lorentz and Marcinkiewicz spaces 
(for more general results of such a sort see \cite{astashkin-curbera-3}).

Recall \cite{astashkin-curbera-3} that a function $\varphi(t)$ defined on $[0,1]$ 
satisfies the \textit{$\Delta^2$-condition} (briefly, $\varphi\in \Delta^2$) if
it is nonnegative, increasing, concave, and there exists $C>0$ such
that $\varphi(t)\le C{\cdot} \varphi(t^2)$ for all $0<t\le 1.$ By 
\cite[Corollary~3.5]{astashkin-curbera-3}, if $\varphi\in \Delta^2$, then
$\mathcal{M}(\Lambda(\varphi))=\mathrm{Sym}(\Lambda(\varphi))$ and 
$\mathcal{M}(M(\varphi))=\mathrm{Sym}(M(\varphi))$.
Moreover, it is known \cite[Example~2.15 and Theorem~4.1]{astashkin-curbera-1} that 
$\mathrm{Sym}(\Lambda(\varphi))=\Lambda(\psi)$ (resp. 
$\mathrm{Sym}(M(\varphi))=M(\psi)$), where $\psi'(t)=\varphi'(t)\log^{1/2}(e/t),$
whenever $\log^{1/2}(e/t)\in \Lambda(\varphi)$ (resp. $\log^{1/2}(e/t)\in M(\varphi)$).
Therefore, we get 

\begin{corollary}
\label{Cor2}
Let  $\varphi\in \Delta^2$ and $\log^{1/2}(e/t)\in \Lambda(\varphi)$ 
(resp. $\log^{1/2}(e/t)\in M(\varphi)$). If $w$ is a positive 
measurable function on $[0,1]$ satisfying condition \eqref{Inequ2}, then 
the sequence $\{r_k\}$ is equivalent in the space $\Lambda(\varphi)(w)$ 
(resp. $M(\varphi)(w)$) to the unit vector
basis in $\ell_2$ if and only if $w\in \Lambda(\psi)$ (resp. $w\in M(\psi)$), where
$\psi'(t)=\varphi'(t)\log^{1/2}(e/t)$.
\end{corollary}


In particular, if $0<p\le 2,$ the sequence $\{r_k\}$ is equivalent in the Zygmund space ${\rm Exp}\,L^p(w)$ 
to the unit vector basis in $\ell_2$  if and only if $w\in {\rm Exp}\,L^q$, where
$q=2p/(2-p)$ (here, we set ${\rm Exp}\,L^\infty=L_\infty).$

\vspace{2 mm}

\section{Rademacher orthogonal projection in weighted spaces}

\begin{proposition}
\label{prop1}
 Let $E$ be a Banach function lattice on $[0,1]$ that is isometrically embedded into $E''$, 
$L_\infty\subset E\subset L_1$. Then the projection
$P$ defined by \eqref{Inequ1} is bounded in $E$ if and only if there are constants $C_1$ and $C_2$
such that for all $a=(a_k)\in\ell_2$
\begin{equation}
\label{Inequ3}
\Big\|\bc a_kr_k\Big\|_E\le C_1\|a\|_{\ell_2}
\end{equation}
and
\begin{equation}
\label{Inequ4}
\Big\|\bc a_kr_k\Big\|_{E'}\le C_2\|a\|_{\ell_2}.
\end{equation}
\end{proposition}
\begin{proof} 
 Firstly, assume that inequalities \eqref{Inequ3} and \eqref{Inequ4} hold. Then, denoting, as above,
$c_k(f):=\int_0^1 f(u)r_k(u)\,du,$ $k=1,2,\dots,$ for every $n\in\mathbb{N}$, by \eqref{Inequ4},
we have
$$
\sum_{k=1}^n c_k(f)^2 =\int_0^1f(u)\sum_{k=1}^n
c_k(f)r_k(u)\,du
\le  \|f\|_E\Big\|\sum_{k=1}^n c_k(f)r_k\Big\|_{E'}
\le C_2\|f\|_E\Big(\sum_{k=1}^n c_k(f)^2\Big)^{1/2},
$$ 
whence
$$
\Big(\bc c_k(f)^2\Big)^{1/2}\le C_2\|f\|_E,\;\;f\in E.$$
Therefore, by \eqref{Inequ3}, we obtain
$$
\|Pf\|_E\le C_1\Big(\bc c_k(f)^2\Big)^{1/2}\le C_1C_2\|f\|_E$$ for all
$f\in E.$

Conversely, suppose that the projection $P$ is bounded in $E.$
Let us consider the following sequence of finite dimensional operators
$$
P_nf(t):=\sum_{k=1}^{n} c_{k}(f) r_{k}(t),\;\;n\in\mathbb{N}.
$$
Clearly, $P_n$ is bounded in $E$ for every $n\in\mathbb{N}.$
Furthermore, by assumption, the series $\sum_{k=1}^{\infty} c_{k}(f) r_{k}$
converges in $E$ for each $f\in E.$ Therefore, by the Uniform Boundedness Principle,
\begin{equation}
\label{Inequ4.5}
\|P_n\|_{E\to E}\le B, \;\;n\in\mathbb{N}.
\end{equation}
Moreover, since $L_\infty\subset E\subset L_1,$ then $L_\infty\subset E'\subset L_1$
as well, and hence, by the $L_1$-Khintchine inequality \eqref{EQ: L_1-inequaily},
$$
\Big\|\bc a_kr_k\Big\|_E\ge c\|a\|_{\ell_2}\;\;\;\mbox{and}\;\;\; \Big\|\bc
a_kr_k\Big\|_{E'}\ge c\|a\|_{\ell_2}.$$ 
Therefore, for all $f\in E,$ $n\in \mathbb{N}$ and $a_k\in\mathbb R,$ $k=1,2,\dots,n,$ we have
\begin{eqnarray*}
\int_0^1f(t)\cdot \sum_{k=1}^n a_kr_k(t)\,dt &=&\sum_{k=1}^n
a_kc_k(f)\le\|a\|_2\Bigl(\sum_{k=1}^nc_k(f)^2\Bigr)^{1/2}\\
&\le & c^{-1}\|a\|_{\ell_2}\cdot\|P_nf\|_E\le Bc^{-1}\|a\|_{\ell_2}\cdot\|f\|_E.
\end{eqnarray*}
Taking the supremum over all $f\in E,$ $\|f\|_E\le1,$ we get
\begin{equation*}
\Bigl\|\sum_{k=1}^n
a_kr_k\Bigr\|_{E'}\le Bc^{-1}\|a\|_{\ell_2},\;\;n\in\mathbb{N}.
\end{equation*}
Applying the latter inequality to Rademacher sums $\sum_{k=n}^m a_kr_k,$ $1\le n<m,$
with $a=(a_k)_{k=1}^\infty\in \ell_2,$ we deduce that the series
$\bc a_kr_k$ converges in the space $E'$ and
\begin{equation*}
\label{eq17} 
\Big\|\bc a_kr_k\Big\|_{E'}\le Bc^{-1}\|a\|_{\ell_2}.
\end{equation*}
Thus, \eqref{Inequ4} is proved. Let us prove similar inequality for $E.$

By Fubini theorem and \eqref{Inequ4.5}, for arbitrary $f\in E$, $g\in E'$ and every 
$n\in \mathbb{N}$ we have
$$
\int_0^1f(u)\cdot\sum_{k=1}^n c_k(g)r_k(u)\,du= \int_0^1 g(t)\cdot
\sum_{k=1}^nc_k(f)r_k(t)\,dt\le\|P_nf\|_E\|g\|_{E'}\le B\|f\|_E\|g\|_{E'},$$
whence
$$
\Big\|\sum_{k=1}^n c_k(g)r_k\Big\|_{E'}\le B\|g\|_{E'},\;\;n\in\mathbb{N}.$$ 
Applying this inequality instead of \eqref{Inequ4.5}, as above, we get
\begin{equation*}
\label{eq18} \Big\|\sum_{k=1}^n a_kr_k\Big\|_{E''}\le Bc^{-1}\|a\|_{\ell_2}.
\end{equation*}
Since $L_\infty\subset E$ and $E$ is isometrically embedded into $E''$, from the last inequality
it follows that
\begin{equation*}
\label{eq18} \Big\|\sum_{k=1}^n a_kr_k\Big\|_{E}\le Bc^{-1}\|a\|_{\ell_2}
\end{equation*}
for all $n\in \mathbb{N}.$ Hence, if $a=(a_k)_{k=1}^\infty\in \ell_2,$
the series $\bc a_kr_k$ converges in $E$ and
\begin{equation*}
\label{eq18} \Big\|\sum_{k=1}^\infty a_kr_k\Big\|_{E}\le Bc^{-1}\|a\|_{\ell_2}.
\end{equation*}
Thus, inequality \eqref{Inequ3} holds, and the proof is complete.
\end{proof}

From Proposition \ref{prop1}, Corollary \ref{prop2} and Theorem \ref{Th1} we obtain 
the following results. 

\begin{theorem}
\label{Cor4}
Let a symmetric space $X$ and a positive measurable function $w$
on $[0,1]$ satisfy condition \eqref{Inequ2}. 
Then, the projection $P$ defined by \eqref{Inequ1} is bounded in $X(w)$ if and only if 
$G\subset X\subset G',$ $w\in \lx$ and $1/w\in {\mathcal{M}(X')}.$

In particular, $P$ is bounded in $X(w)$ whenever 
$w^*(t)\loget\in X''$ and $(1/w)^*(t)\loget\in X'.$
\end{theorem}

As above, the result can be somewhat refined for Lorentz and Marcinkiewicz spaces
whose fundamental function satisfies the $\Delta^2$-condition.

\begin{corollary}
\label{Cor5}
Let  $\varphi\in \Delta^2$ and let $w$ be a positive measurable function on $[0,1]$ 
satisfying condition \eqref{Inequ2} for $X=\Lambda(\varphi)$ (resp. $X=M(\varphi)$). Then
the projection $P$ defined by \eqref{Inequ1} is bounded in $\Lambda(\varphi)(w)$
(resp. $M(\varphi)(w)$) if and only if $G\subset \Lambda(\varphi)\subset G'$,
$w\in \Lambda(\psi)$ and $1/w\in {\mathcal{M}(M(\tilde\varphi))}$ 
(resp. $G\subset M(\varphi)\subset G'$, $w\in M(\psi)$ and $1/w\in {\mathcal{M}(\Lambda(\tilde\varphi))}),$ 
where $\psi'(t)=\varphi'(t)\log^{1/2}(e/t)$ and $\tilde\varphi(t)=t/\varphi(t).$
\end{corollary}


\begin{remark}
 It is easy to see that the orthogonal projection $P$ is bounded in the space $X(w)$ if and only if
the projection 
$$
P_wf(t):=\sum_{k=1}^{\infty} \int_0^1 f(s)r_k(s)\,\frac{ds}{w(s)}\cdot r_k(t)w(t),\;\;0\le t\le 1,$$
(on the subspace $[r_kw]$) is bounded in $X.$
\end{remark}

\vspace{2 mm}

\section{Example of a function from $\mathcal{M}(L_1)\setminus\mathrm{Sym}\,(L_1)$}

Answering a question from \cite{curbera1}, we present here a concrete example of 
a function $f\in\mathcal{M}(L_1)$, which does not belong to the symmetric 
kernel $\mathrm{Sym}\,(L_1)$, that is, 
$$
\int_0^1 f^*(t) \log^{1/2}(e/t)\,dt=\infty.$$
Since the latter space is symmetric, it is sufficient to find a function 
$f\in\mathcal{M}(L_1),$ for which
there exists a function $g\not\in\mathcal{M}(L_1)$ equimeasurable with $f$.
We will look for $f$ and $g$ in the form
\begin{equation}
\label{eq181}
f=\bc\alpha_k\chi_{B_k},\;\;g=\bc\alpha_k\chi_{D_k},
\end{equation}
where $\{B_k\}$ and $\{D_k\}$ are sequences of pairwise disjoint subsets of $[0,1],$
$m(B_k)=m(D_k),$ $\alpha_k\in\mathbb{R},$ $k=1,2,\dots$ Next, we will make use of some
ideas of the paper \cite{curbera}.

Let $n=2^m$ with $m\in\mathbb{N}$ and let $J$ be a subset of $\{1,2,\dots,2^n\}$ 
with cardinality $n$. We define the set $A=\bigcup_{j\in J}\Delta_n^j$
associated with $J$ (as above, $\Delta^j_n$ are the dyadic intervals of $[0,1]).$
Clearly, $m(A)=n2^{-n}.$ 

For arbitrary sequence $(b_i)\in\ell_2$ we have
\begin{equation}
\label{eq182}
\Big\|\chi_A\sum_{i=1}^\infty
b_ir_i\Big\|_1\le\Big\|\chi_A\sum_{i=1}^n
b_ir_i\Big\|_1+\|\chi_A\sum_{i=n+1}^\infty b_ir_i\Big\|_1.
\end{equation}
Firstly, we estimate the tail term from the right hand side of this inequality. 
It is easy to see that the functions
$$\chi_A(t)\cdot\sum_{i=n+1}^\infty b_ir_i(t)\;\;\mbox{and}\;\;
\chi_{[0,n2^{-n}]}(t)\cdot\sum_{i=n+1}^\infty b_ir_i(t)$$
are equimeasurable on $[0,1]$ and
$$
\chi_{[0,n2^{-n}]}(t)\sum_{i=n+1}^\infty b_ir_i(t)=\sum_{i=n+1}^\infty 
b_ir_{i+m-n}(n2^{-n}t),\;\;0<t\le 1.$$
Therefore,  
\begin{equation}
\label{eq183}
\Big\|\chi_A\sum_{i=n+1}^\infty b_ir_i\Big\|_1 =
\Big\|\chi_{[0,n2^{-n}]}\sum_{i=n+1}^\infty b_ir_i\Big\|_1=
n2^{-n}\Big\|\sum_{i=n+1}^\infty b_ir_{i+m-n}\Big\|_1\le n2^{-n}\Big(\sum_{i=n+1}^\infty b_i^2\Big)^{1/2}.
\end{equation}

Now, choosing a set $A$ in a special way, estimate the first  term from the right hand side
of \eqref{eq182}. Denote by $\varepsilon_{ij}^n$ the value of the function 
$r_i,$ $i=1,2,\dots,n,$ on the interval $\Delta^j_n,$ $1\le j\le
2^n.$ Since $n=2^m$, we can find a set $J_1(n)\subset \{1,2,\dots,2^n\}$,
${\rm card}\,J_1(n)=n,$ such that the $n\times n$ matrix 
$n^{-1/2}\cdot (\varepsilon_{ij}^n)_{1\le i\le n,j\in J_1(n)}$ is orthogonal.
Then, if $c_j:=n^{-1/2}\sum_{i=1}^n \varepsilon_{ij}^nb_i$, $j\in J_1(n)$, we have 
$\|(c_j)_{j\in J_1(n)}\|_{\ell_2}=\|(b_i)_{i=1}^n\|_{\ell_2}$. Therefore, 
setting $B(n):=\bigcup_{j\in J_1(n)}\Delta_n^j,$ we obtain
\begin{eqnarray*}
\Big\|\chi_{B(n)}\sum_{i=1}^n b_ir_i\Big\|_1 & = &
\Big\|\sum_{j\in J_1(n)}\Big(\sum_{i=1}^n
b_ir_i\Big)\chi_{\Delta^j_n}\Big\|_1=\Big\|\sum_{{j\in J_1(n)}}\sum_{i=1}^n 
\varepsilon_{ij}^nb_i\cdot\chi_{\Delta^j_n}\Big\|_1\\
& = & n^{1/2}\Big\|\sum_{{j\in J_1(n)}}c_j\chi_{\Delta^j_n}\Big\|_1=
n^{1/2}2^{-n}\sum _{{j\in J_1(n)}}|c_j|\le n2^{-n}\|(b_i)_{i=1}^n\|_{\ell_2}.
\end{eqnarray*}
Combining this inequality with \eqref{eq182}, \eqref{eq183} for $A=B(n)$ and 
\eqref{EQ: L_1-inequaily}, by definition of the norm
in the space $\mathcal{M}(L_1)$, we have 
\begin{equation}
\label{eq184}
\|\chi_{B(n)}\|_{\mathcal{M}(L_1)}\le 2\sqrt{2} n2^{-n}.
\end{equation}

Let $\{n_k\}_{k=1}^\infty$ be an increasing sequence of positive integers, $n_k=2^{m_k},$
$m_k\in\mathbb{N},$ satisfying the condition  
\begin{equation}
\label{eq185}
n_k^{1/8}\ge 2^{n_1+\dots +n_{k-1}},\;\;k=2,3,\dots
\end{equation}
At first, we construct a sequence of sets $\{B_k\}$. Setting $J_1^1:=J_1(n_1)$ and $B_1:=B(n_1),$
in view of \eqref{eq184} we have 
$$\|\chi_{B_1}\|_{\mathcal{M}(L_1)}\le 2\sqrt{2} n_12^{-n_1}.$$.

To define $B_2$, we take for $I_1$ any interval $\Delta^j_{n_1}$ such that $j\not\in J_1^1.$
Now, we can choose a set $J_1^2\subset \{1,2,\dots,2^{n_1+n_2}\}$ satisfying the conditions:
${\rm card}\,J_1^2=n_2,$ $\Delta^j_{n_1+n_2}\subset I_1$ for every  $j\in J_1^2$ and
the $n_2\times n_2$ matrix 
$n_2^{-1/2}\cdot (\varepsilon_{ij}^{n_1+n_2})_{n_1<i\le n_1+n_2,j\in J_1^2}$ is orthogonal.
We set $B_2:=\bigcup_{j\in J_1^2}\Delta_{n_1+n_2}^j.$ Clearly, $m(B_2)=n_22^{-(n_1+n_2)}$ and
$B_1\cap B_2=\emptyset,$ because of $B_2\subset I_1.$ As in the case of $B(n)$ we have
\begin{eqnarray*}
\Big\|\chi_{B_2}\sum_{i=1}^{n_1+n_2} b_ir_i\Big\|_1 & = &
\Big\|\sum_{j\in J_1^2}\Big(\sum_{i=1}^{n_1+n_2}
b_ir_i\Big)\chi_{\Delta^j_{n_1+n_2}}\Big\|_1\le \Big\|\sum_{j\in J_1^2}\Big(\sum_{i=1}^{n_1}
b_ir_i\Big)\chi_{\Delta^j_{n_1+n_2}}\Big\|_1\\ &+&
\Big\|\sum_{j\in J_1^2}\Big(\sum_{i=n_1+1}^{n_2}
b_ir_i\Big)\chi_{\Delta^j_{n_1+n_2}}\Big\|_1\le \sum_{i=1}^{n_1}|b_i|\|\chi_{B_2}\|_1+
\Big\|\sum_{{j\in J_1^2}}\sum_{i=n_1+1}^{n_1+n_2} 
\varepsilon_{ij}^{n_1+n_2}b_i\cdot\chi_{\Delta^j_{n_1+n_2}}\Big\|_1\\
& \le & (n_1^{1/2}+1)n_22^{-(n_1+n_2)}\|(b_i)_{i=1}^{n_1+n_2}\|_{\ell_2}
\le n_22^{-n_2}\|(b_i)_{i=1}^{n_1+n_2}\|_{\ell_2}.
\end{eqnarray*}
Therefore, from  \eqref{eq182}, \eqref{eq183} and \eqref{EQ: L_1-inequaily} it follows that
$$\|\chi_{B_2}\|_{\mathcal{M}(L_1)}\le \sqrt{2}\left((n_1+n_2)2^{-(n_1+n_2)}+
n_22^{-n_2}\right)\le 2\sqrt{2}n_22^{-n_2}.$$
Proceeding in the same way, we get a sequence $\{B_k\}$ of pairwise disjoint subsets of $[0,1]$
such that $m(B_k)=n_k2^{-(n_1+\dots +n_k)}$ and
\begin{equation}
\label{eq186}
\|\chi_{B_k}\|_{\mathcal{M}(L_1)}\le 2\sqrt{2}n_k2^{-n_k},\;\;k=1,2,\dots
\end{equation}

Now, define the sets $D_k$, $k=1,2,\dots$ Select a set $J_2^1\subset \{1,2,\dots,2^{n_1}\}$,
${\rm card}\,J_2^1=n_1,$ such that each column of the $n_1\times n_1$ matrix 
$(\varepsilon_{ij}^{n_1})_{1\le i\le n_1,j\in J_2^1}$ has exactly one entry equal to $-1$
and the rest are equal to 1. Setting $D_1:=\bigcup_{j\in J_2^1}\Delta_{n_1}^j,$ we have
$m(D_1)=n_12^{-n_1}$. Furthermore, from the inequality 
$\|n_1^{-1/2}\sum_{i=1}^{n_1}r_i\|_1\le 1$ (see \eqref{EQ: L_1-inequaily})
and the definition of $D_1$ it follows that
\begin{eqnarray*}
\|\chi_{D_1}\|_{\mathcal{M}(L_1)} & \ge & \Big\|\sum_{j\in
J_2^1}\Big({n_1^{-1/2}}\sum_{i=1}^{n_1}
{r_i}\Big)\chi_{\Delta_{n_1}^j}\Big\|_1\\
&=& \Big\|\sum_{j\in J_2^1}\Big({n_1^{-1/2}}\sum_{i=1}^{n_1}
{\varepsilon_{ij}^{n_1}}\Big)\chi_{\Delta_{n_1}^j}\Big\|_1\\
& = & (n_1^{1/2}-2n_1^{-1/2}){n_1}{2^{-n_1}}\ge
\frac{1}{2}n_1^{3/2}{2^{-n_1}}
\end{eqnarray*}
if $n_1$ is large enough.

Similarly, we can define the set $D_2.$ Let $I_2$ be any interval $\Delta^j_{n_1}$ 
with $j\not\in J_2^1.$ Choose the set $J_2^2\subset \{1,2,\dots,2^{n_1+n_2}\}$
such that ${\rm card}\,J_2^2=n_2,$ $\Delta^j_{n_1+n_2}\subset I_2$ for every  $j\in J_2^2$ and 
each column of the $n_2\times n_2$ matrix 
$(\varepsilon_{ij}^{n_1+n_2})_{n_1< i\le n_1+n_2,j\in J_2^2}$ has exactly one entry equal to $-1$
and the rest are equal to 1. Then, if $D_2:=\bigcup_{j\in J_2^2}\Delta_{n_1+n_2}^j,$ 
then $m(D_2)=n_22^{-(n_1+n_2)}$ and $D_1\cap D_2=\emptyset.$ Moreover, we have
\begin{eqnarray*}
\|\chi_{D_2}\|_{\mathcal{M}(L_1)} & \ge & \Big\|\sum_{j\in
J_2^2}\Big({n_2^{-1/2}}\sum_{i=n_1+1}^{n_1+n_2}
{r_i}\Big)\chi_{\Delta_{n_1+n_2}^j}\Big\|_1\\
&=& \Big\|\sum_{j\in J_2^2}\Big({n_2^{-1/2}}\sum_{i=n_1+1}^{n_1+n_2}
{\varepsilon_{ij}^{n_1+n_2}}\Big)\chi_{\Delta_{n_1+n_2}^j}\Big\|_1\\
& = & (n_2^{1/2}-2n_2^{-1/2}){n_2}{2^{-(n_1+n_2)}}\ge
\frac{1}{2}n_2^{3/2}{2^{-(n_1+n_2)}}.
\end{eqnarray*}
Arguing in the same way, we construct a sequence $\{D_k\}$ of pairwise disjoint subsets of $[0,1]$
such that $m(D_k)=n_k2^{-(n_1+\dots +n_k)}$ and 
\begin{equation}
\label{eq187}
\|\chi_{D_k}\|_{\mathcal{M}(L_1)} \ge \frac{1}{2}n_k^{3/2}{2^{-(n_1+\dots +n_k)}},
k=1,2,\dots
\end{equation}

Since $m(B_k)=m(D_k)$, $k=1,2,\dots$, the functions $f$ and $g$ defined by
\eqref{eq181} are equimeasurable ones for arbitrary $\alpha_k\in\mathbb{R},$ $k=1,2,\dots$
Setting $\alpha_k= 2^{n_k}n_k^{-5/4},$ by \eqref{eq186}, we obtain
$$
\|f\|_{\mathcal{M}(L_1)}\le \bc\alpha_k\|\chi_{B_k}\|_{\mathcal{M}(L_1)}\le
2\sqrt{2}\bc n_k^{-1/4}<\infty,$$
because of $n_k=2^{m_k},$ $m_1<m_2<\dots$ Thus, $f\in \mathcal{M}(L_1).$

On the other hand, since $\mathcal{M}(L_1)$ is a Banach function lattice, for every $k=1,2,\dots$ 
from \eqref{eq187} and \eqref{eq185} it follows that
$$
\|g\|_{\mathcal{M}(L_1)}\ge \alpha_k\|\chi_{D_k}\|_{\mathcal{M}(L_1)}\ge
\frac{1}{2}n_k^{1/4}{2^{-(n_1+\dots +n_{k-1})}}\ge \frac{1}{2}n_k^{1/8}.$$
Hence, $g\not\in \mathcal{M}(L_1).$

\vspace {2 mm}

\noindent
Department of Mathematics
and Mechanics \\
Samara State University\\
Acad. Pavlov, 1 \\
443011 Samara \\
Russian Federation

\vspace{1 mm}
\noindent
E-mail: astash@samsu.ru


\vspace{3 mm}

\begin{thebibliography}{99}

\bibitem{astashkin3}
S. V. Astashkin, {\it Systems of random variables equivalent in
distribution to the Rademacher system and $\mathcal{K}$-closed
representability of Banach pairs}, Matem. sb. \textbf{191}(2000). no.~6, 3--30
(Russian); English transl. in Sb. Math. \textbf{191}(2000), 779--807.

\bibitem{As09} S. V.~Astashkin, {\it Rademacher functions in symmetric spaces},
Sovrem. Mat. Fundam. Napravl.,  \textbf{32}(2009), 3--161 (Russian);  English
transl. in  J. Math. Sci. (N.Y.) (6),  \textbf{169}(2010), 725--886.

\bibitem{astashkin-curbera-1}
S. V. Astashkin and G. P. Curbera, {\it Symmetric kernel of Rademacher
multiplicator spaces}, J. Funct. Anal. \textbf{226}(2005), 173--192.

\bibitem{astashkin-curbera-2}
S. V. Astashkin and G. P. Curbera,
{\it Rademacher multiplicator spaces equal to $L^\infty$}, Proc. Amer. Math. Soc.
\textbf{136}(2008), 3493--3501.

\bibitem{astashkin-curbera-3}
S. V. Astashkin and G. P. Curbera, {\it Rearrangement invariance of Rademacher
multiplicator spaces}, J. Funct. Anal. \textbf{256}(2009), 4071--4094.

\bibitem{astashkin-curbera-5} S. V. Astashkin and G. P. Curbera, {\it A weighted Khintchine inequality}, 
Revista Mat. Iberoam. \textbf{30}(2014), no. 1, 237--246.

\bibitem{BS}
C.~Bennett and R.~Sharpley, {\it Interpolation of Operators},
Pure and Applied Mathematics, Vol. 119, Academic Press, Boston,
1988.

\bibitem{curbera-1}
G. P. Curbera, {\it Operators into $L^1$ of a vector measure and applications to 
Banach lattices}, Math. Ann. \textbf{293}(1992), 317--330.

\bibitem{curbera}
G. P. Curbera, {\it A note on function spaces generated by Rademacher
series}, Proc. Edinburgh. Math. Soc. \textbf{40}(1997),
119--126.

\bibitem{curbera1}
G. P. Curbera, {\it How summable are Rademacher series?}, 
Operator Theory: Adv. and Appl. \textbf{201}(2009), 135--148.


\bibitem{DJT}
J.~Diestel, H.~Jarchow and A.~Tonge, {\it Absolutely Summing
Operators}, Cambridge University Press, Cambridge, 1995.

\bibitem{JMST}
W. B.~Johnson, B.~Maurey, G.~Schechtman and L.~Tzafriri,
{\it Symmetric structures in Banach spaces}, Mem. Amer. Math.
Soc. No. 217, 1979.



\bibitem{Kh} A. Khiintchine, {\it \"{U}ber dyadische Bruche,} Math. Zeit.
\textbf{18}(1923), 109--116.


\bibitem{KPS}
S. G.~Krein, Ju. I.~Petunin and E. M.~Semenov,
{\it Interpolation of Linear Operators}, AMS Translations of
Math. Monog., 54, Providence, 1982.

\bibitem{LT}
J.~Lindenstrauss and L.~Tzafriri, {\it Classical Banach Spaces
II}, vol. 97, Springer-Verlag, Berlin, 1979.

\bibitem{lorentz}
G. G. Lorentz,  {\it Relations between function spaces}, Proc. Amer.
Math. Soc. \textbf{12}(1961), 127--132.

\bibitem{MSch}
V. D.~Milman and G.~Schechtman,
{\it Asymptotic theory of finite dimensional normed spaces}, Lecture Notes
in Mathematics, vol. 1200, Springer-Verlag, Berlin, 1986.


\bibitem{PZ30} R. E. A. C. Paley and A. Zygmund, {\it On some series of functions. I, II}, 
Proc. Camb. Phil. Soc. {\bf 26}(1930), 337--357, 458--474.


\bibitem{Pisier}
G.~Pisier, {\it Factorization of linear operators and geometry of
Banach spaces}, Amer. Math. Soc., Providence, RI, CBMS 60, 1986.

\bibitem{RS}
V. A.~Rodin and E. M.~Semyonov, {\it Rademacher series in
symmetric spaces}, Anal. Math. \textbf{1}(1975), no.~3, 207--222.

\bibitem{RS79} V. A.~Rodin and E. M.~Semenov,
{\it The complementability of a subspace that is generated by the
Rademacher system in a symmetric space}, Funktsional. Anal.
i~Prilozhen. (2),  \textbf{13}(1979), 91--92 (Russian); English
transl. in Functional Anal. Appl. \textbf{13}(1979), no.~2, 150--151.

\bibitem {Szarek} S. J.~Szarek, {\it On the best constants in
the Khinchin inequality}, Studia Math. \textbf{58}(1976), no.~2, 197--208.

\bibitem{veraar}
M. Veraar, {\it On Khintchine inequalities with a weight}, Proc.
Amer. Math. Soc. \textbf{138}(2011), 4119--4121.

\bibitem{Zy59} A. Zygmund, {\it Trigonometric Series. 2nd ed. Vol. I}, Cambridge University Press, 
New York 1959.


\end{thebibliography}
\end{document}